\documentclass[a4paper]{article}
\usepackage[utf8]{inputenc}
\usepackage{graphicx}
\usepackage{flushend}
\usepackage{amsmath}
\usepackage{amsthm}
\usepackage{nccmath}
\usepackage{amssymb}
\usepackage{enumerate}

\def\n{\nabla}

\def\M{\mathbb{M}}
\def\r{\mathbb{R}}
\def\d{\mathbb{D}}
\def\c{\mathbb{C}}
\def\R{\mathbb{R}}
\def\t{\tau}
\def\k{\kappa}
\def\h{\mathbb{H}}
\def\s{\mathbb{S}}
\def\mx{\mathfrak{X}}
\def\mt{\mathbf{T}}
\def\la{\langle}
\def\ra{\rangle}
\def\nt{\mathbf{t}}
\def\te{\theta}

\def\mer{\mathbb{M}^2 (\kappa)\times \mathbb{R}}
\def\mr{\mathbb{M}^2 \times \mathbb{R}}
\def\hr{\mathbb{H}^2\times \mathbb{R}}

\newcommand{\zb}{\bar{z}}
\newcommand{\hmf}{\mathbb{E}(\kappa , \tau)}

\newcommand{\norm}[1]{\left\vert #1 \right\vert}
\newcommand{\abs}[1]{\left\vert #1 \right\vert}
\newcommand{\set}[1]{\left\{#1\right\}}
\newcommand{\meta}[2]{\langle #1,#2 \rangle }

\theoremstyle{plain}
 \newtheorem{definition}{Definition}
 \newtheorem{theorem}{Theorem}
 \newtheorem{prop}{Proposition}
 \newtheorem{remark}{Remark}
 \newtheorem{lema}[prop]{Lema}
 
 \newtheorem{corollary}{Corollary}

\begin{document}
\title{The Abresch-Rosenberg Shape Operator\\and applications}
\author{José M. Espinar and Haimer A. Trejos}
\date{}
\maketitle

\vspace{.5cm}

\noindent $\mbox{}^\dag$ Instituto de Matematica Pura y Aplicada, 110 Estrada Dona
Castorina, Rio de Janeiro 22460-320, Brazil; e-mail: jespinar@impa.br

\noindent $\mbox{}^\ddag$ Instituto de Matematica Pura y Aplicada, 110 Estrada Dona
Castorina, Rio de Janeiro 22460-320, Brazil; e-mail: aletrejosserna@gmail.com

\vspace{.3cm}

\begin{abstract}
There exists a holomorphic quadratic differential defined on any $H-$ surface immersed in the homogeneous space $\hmf$ given by U. Abresch and H. Rosenberg \cite{Rosen,Rosen1}, called the Abresch-Rosenberg differential. However, there were no Codazzi pair on such $H-$surface associated to the  Abresch-Rosenberg differential when $\t \neq 0$. The goal of this paper is to find a geometric Codazzi pair defined on any $H-$surface in $\hmf$, when $\t \neq 0$, whose $(2,0)-$part is the Abresch-Rosenberg differential. 

In particular, this allows us to compute a Simons' type formula for $H-$surfaces in $\hmf$. We apply such Simons' type formula, first, to study the behavior of complete $H-$surfaces $\Sigma$ of finite Abresch-Rosenberg total curvature immersed in $\hmf$. Second, we estimate the first eigenvalue of any Schr\"{o}dinger operator $L= \Delta + V$, $V$ continuous, defined on such surfaces. Finally, together with the Omori-Yau's Maximum Principle, we classify complete $H-$surfaces in $\hmf$, $\t \neq 0$, satisfying a lower bound on $H$ depending on $\kappa$ and $\tau$.
\end{abstract}



\section{Introduction.}

W. P. Thurston proved that the building blocks of the Geometrization Conjecture are given by eight maximal model geometries, in other words, a maximal model geometry is a simply connected smooth manifold $\mathcal M$ together with a transitive action of the Lie group $G$ on $\mathcal M$, which is maximal among groups acting smoothly and transitively on $\mathcal M $, with compact stabilizers. Such maximal model geometries can be classified according to the dimension of its isometry group. If the dimension is 6, they correspond to the space forms $\mathbb{M}^{3}(\k)$. When the dimension is 3, the manifold has the geometry of the Lie group ${\rm Sol}_{3}$, and when the dimension is 4, they correspond to a 2-parameter family $\k,\t \in \r$, $\k -4 \t ^2 \neq 0$, of manifolds denoted by $\hmf$. These manifolds correspond to the product spaces $\mer$, when $\k \neq 0, \t=0 $, where $\mathbb{M}^{2}(\k)$ is the simply connected surface of constant curvature $\k$. The Heisenberg space ${\rm Nil}_{3}$, when $\k=0,\t \neq 0$. The covering space of the linear group ${\rm PSl}_{2}(\r)$, when $\k< 0, \t \neq 0$. The Berger sphere $\mathbb{S}^{3}_{B}(\k,\t)$, when $\k>0, \t \neq 0$. It is known that $\hmf$ is a Riemannian submersion over $\mathbb{M}^{2}(\k)$ with fiber bundle curvature $\t$ and the fibers are integral curves of a unit Killing field defined in $\hmf$.  \\

Constant mean curvature surfaces $\Sigma$, in short, $H-$surfaces, immersed in homogeneous 3-spaces $\hmf$ are receiving an impressive number of contributions since U. Abresch and H. Rosenberg \cite{Rosen,Rosen1} showed the existence of a holomorphic quadratic differential, the Abresch-Rosenberg differential, on such surfaces. Hence, they we able to extend the well-known Hopf's Theorem, i.e., they classified the topological spheres with constant mean curvature as the rotationally symmetric ones in $\hmf$. \\

When $\Sigma$ is isometrically immersed in a space form $\mathbb{M}^{3}(\kappa)$, the second fundamental form defines a bilinear symmetric tensor that satisfies the Codazzi equation. The Codazzi equation is fundamental to ensure that the usual Hopf differential is holomorphic on any $H-$surface in a space form; nevertheless, the above fact is no longer true when the surface is isometrically immersed in a homogeneous 3-manifold $\hmf$. However, J.A. Aledo, J.M. Espinar and J.A. G\'{a}lvez \cite{AEG3} obtained a geometric Codazzi pair $(I,II_{S})$ on any $H-$surface in $\mer$ so that the $(2,0)$-part of $II_{S}$ with respect to the conformal structure given by the first fundamental form $I$, is the Abresch-Rosenberg differential. However, up to now, despite the existence of a holomorphic quadratic differential on any $H-$surface in $\hmf$, $\tau \neq 0$, there were no natural (or geometric) Codazzi pair on such $H-$surfaces.  \\



In \cite{Simm}, J. Simons computed the Laplacian of the squared norm of the second fundamental form of a minimal submanifold in $\s ^n$. Nowadays, such formula is known as the {\bf Simons' formula} and it is a powerful tool to obtain classification theorems for minimal hypersurfaces. Later, Simons' formula has been extensively generalized by many authors in different situations (cf. \cite{Batista,Chern, CCK,FR}). In \cite{Yau}, S.Y. Cheng and  S.T. Yau computed an abstract version of the Simons' formula, that is, they computed the Laplacian of the squared norm of any bilinear symmetric Codazzi $(2,0)-$tensor defined on a Riemannian manifold $(\mathcal M^{n} ,g)$. Specifically, if $\Phi$ is a bilinear symmetric Codazzi $(2,0)-$tensor whose local expression in a orthonormal coframe $\{\omega_1 , \ldots , \omega_n \}$ is given by $\displaystyle{\Phi = \sum_{i,j} \phi_{ij} \, \omega_{i} \otimes \omega_{j}}$, then the Laplacian of $|\Phi|^{2}$ is given by 

\begin{equation}\label{Simmons}
\frac{1}{2} \Delta|\Phi|^{2} = \sum_{i,j,k} (\phi_{ij,k})^{2} + \sum_{i} \lambda_{i} ({\rm tr}(\Phi ))_{ii} + \frac{1}{2} \sum_{i,j} R_{ijij}(\lambda_{i} - \lambda_{j})^{2},
\end{equation}where $R_{ijij}$ is the Riemann curvature tensor associated to the metric $g$,  ${\rm tr} (\Phi )$ is the trace operator, $\lambda _i$, $i =1, \ldots , n$, are the eigenvalues of $\Phi$ for an orthonormal frame $\{e_1 , \ldots , e_n \}$ and $\Phi$ is a Codazzi tensor, that means, $\phi_{ij,k} = \phi_{ik,j}$.\\

A key ingredient to compute \eqref{Simmons} is that $\Phi$ is Codazzi. So, when the hypersurface is isometrically immersed in a space form $\mathbb{M}^{n}(\k)$, the second fundamental form  satisfies the Codazzi equation, hence we  can recover Simons' original formula. As said above, when $\t = 0$, the existence of a holomorphic quadratic differential  on a $H-$surface in $\mer$ implies the existence of a pair $(I,II_{S})$ defined on this $H-$surface, where $I$ is the induced metric of $\mer$ and $II_{S}$ is a symmetric Codazzi $(2,0)-$tensor with respect to the connection induced by $I$. These type of pairs $(I,II_{S})$ are known as Codazzi pairs (cf. \cite{AEG3,Mi}) and this property of $(I,II_{S})$ were used by M. Batista in \cite{Batista} to compute a Simons' type formula for $H-$surfaces in $\mer$.\\

One of the main points of this paper is to obtain a geometric Codazzi pair $(I,II_{AR})$, where $II_{AR}$ is a symmetric $(2,0)-$tensor, that we call the {\bf Abresch-Rosenberg fundamental form}, on any $H-$surface in $\hmf$ whose $(2,0)-$part with respect to the conformal structure induced by $I$ is the Abresch-Rosenberg differential. In other words, $(I,II_{AR})$ is a Codazzi pair on any $H-$surface. In particular, this allows us to compute a Simons' type formula for $H-$surfaces in $\hmf$. We apply such formula, first, to study the behavior of complete $H-$surfaces $\Sigma$ of finite Abresch-Rosenberg total curvature immersed in $\hmf$, i.e., those $H-$surface whose $L^2 -$norm of the traceless Abresch-Rosenberg fundamental form is finite. Observe that complete $H-$surfaces $\Sigma \subset \r ^{3}$ of finite total curvature, that is, those whose $L^{2}-$norm of its traceless second fundamental form is finite, are of capital importance on the comprehension of $H-$surfaces in $\r ^3$. In the case $H=0$, Osserman's Theorem gives an impressive description of them. If $\Sigma $ has constant nonzero mean curvature and finite total curvature, then it must be compact. In our case, we extend the latter result when $H$ is greater than a constant depending only on $\kappa$ and $\tau$. We also estimate the first eigenvalue of any Schr\"{o}dinger operator $L= \Delta + V$, $V$ continuous, defined on $H-$surfaces of finite Abresch-Rosenberg total curvature. Finally, together with the Omori-Yau's Maximum Principle, we classify complete $H-$surfaces (not necessary of finite Abresch-Rosenberg total curvature) in $\hmf$, $\t \neq 0$.\\

\subsection{Outline of the paper}

In Section 2, we set up the notation and we review some of the standard facts on Codazzi pairs.\\ 
 
Section 3 is devoted to the Codazzi pair interpretation of the Abresch-Rosenberg differential and its geometric properties. First, we discuss the known case of $H-$surfaces in a product space $\mathbb{M}^{2}(\kappa) \times \mathbb{R}$. Later, we obtain a geometric Codazzi pair associated to the Abresch-Rosenberg differential on any $H-$surface immersed in $\hmf$ when $\t \neq 0$. Specifically, Lemma \ref{Lem:Key} says
\begin{quote}
{\bf Key Lemma.}  {\it Given a $H-$surface $\Sigma \subset \hmf$, $\tau \neq 0$, consider the symmetric $(2,0)-$tensor given by 
$$II_{AR}(X,Y) = II(X,Y) - \alpha \la \mt_{\te},X \ra \la \mt_{\te},Y \ra + \frac{\alpha \norm{\mt}^{2}}{2}\la X,Y \ra, $$where 
\begin{itemize}
\item $\alpha = \dfrac{\kappa - 4\t^{2}}{2 \sqrt{H^{2} + \t^{2}}}$,
\item $e^{2i\te} = \dfrac{H - i\t}{\sqrt{H^{2} + \t^{2}}}$ and 
\item $\mt_{\te} = \cos\te \, \mt + \sin\te \, J\mt$.
\end{itemize}

Then, $(I, II_{AR})$ is a Codazzi pair with constant mean curvature $H$. Moreover, the $(2,0)-$part of $II_{AR}$ with respect to the conformal structure given by $I$ agrees (up to a constant) with the Abresch-Rosenberg differential.}
\end{quote}

In Section 4, we obtain a Simons' type formula on any $H-$surface in $\hmf$, $\t \neq 0$ (cf. \cite{Batista} when $\t =0$). To do so, we use the Codazzi pair defined on the Key Lemma and \cite{Yau}, 

\begin{quote}
{\bf Theorem \ref{Simfor}.} {\it Let $\Sigma$ be a $H-$surface in $\hmf$. Then, the traceless Abresch-Rosenberg shape operator satisfies
\begin{equation*}
\frac{1}{2} \Delta |S|^{2} =|\n S|^{2}  + 2 K |S|^{2},
\end{equation*}or, equivalently, away from the zeroes of $\abs{S}$,  
\begin{equation*}
\abs{S} \, \Delta \abs{S} - 2 K \abs{S}^2 =  \abs{\nabla \abs{S}}^{2}.
\end{equation*}}
\end{quote}

Section 5 is devoted to complete $H-$surfaces in $\hmf$ of finite Abresch-Rosenberg total curvature, i.e., those $H-$surfaces that satisfy
$$ \int _\Sigma  \abs{S} ^2 < +\infty .$$

We must point out here that the family of complete constant mean curvature surfaces of finite Abresch-Rosenberg total curvature is large. We focus on $H=1/2$ surfaces in $\h ^2 \times \r $ to show this fact. Recall the following result of Fern\'{a}ndez-Mira:

\begin{quote}
{\bf Theorem \cite[Theorem 16]{FM}.} {\it Any holomorphic quadratic differential on an open simply connected Riemann surface is the Abresch-Rosenberg differential of some complete surface $\Sigma$ with $H=1/2$ in $\h ^2 \times \r$. Moreover, the space of noncongruent complete mean curvature one half surfaces in $\h^2 \times \r$ with the same Abresch-Rosenberg differential is generically infinite.}
\end{quote}

We will see that, if we take the disk $\d $ as our open Riemann surface and a holomorphic quadratic differential on $\d$ that extends continuously to the boundary, then the $H=1/2$ surface $\Sigma$ constructed in \cite[Theorem 16]{FM} has finite Abresch-Rosenberg total curvature.

The above examples also show that, despite what happens in $\r ^3$, $H-$surfaces $\Sigma \subset \hmf$ of finite Abresch-Rosenberg total curvature are not necessarily conformally equivalent to a compact surface minus a finite number of points, in particular, $\Sigma$ is not necessarily parabolic. However, we can obtain (cf. \cite{BCF} when $\t =0$):

\begin{quote}
{\bf Theorem \ref{ThConfEquiv}.} {\it Let $\Sigma$ be a complete surface in $\hmf$, $H^{2}+\t^{2} \neq 0$, of finite Abresch-Rosenberg total curvature. Suppose one of the following conditions holds 
\begin{itemize}
\item[1.] $\kappa-4\t^{2}>0$ and $H^{2}+\t^{2} > \frac{\kappa-4\t^{2}}{4} $.
\item[2.] $\kappa-4\t^{2}<0$ and $H^{2}+\t^{2} > -\frac{(\sqrt{5}+2)}{4}(\kappa-4\t^{2}) $.
\end{itemize}

Then, $\Sigma$ is compact.}
\end{quote}

Also, we extend Simons' first stability eigenvalue estimate (cf. \cite{Simm})  to Schr\"{o}dinger operators $L= \Delta + V$ defined on a complete $H-$surface of finite Abresch-Rosenberg total curvature, $H^{2}+\t^{2} \neq 0$, immersed in $\hmf$,

\begin{quote}
{\bf Theorem \ref{Thm:Estimate1}.} {\it Let $\Sigma$ be a complete two-sided $H-$surface in $\hmf$ of  finite Abresch-Rosenberg total curvature and $H^{2} + \t^{2} \neq 0$. Denote by $\lambda_{1}(L)$ the first eigenvalue associated to the Schr\"{o}dinger operator $L := \Delta + V$, $V\in C^0 (\Sigma)$. Then, $\Sigma$ is either an Abresch-Rosenberg surface, a Hopf cylinder or 
\begin{equation*}
\lambda _{1}(L) < - {\rm inf}_{\Sigma}\set{V+2K}.
\end{equation*}}
\end{quote}

In particular, when $L$ is the Stability (or Jacobi) operator, i.e., 
$$ L = \Delta + (\abs{A}^{2} + {\rm Ric} (N)) ,$$where ${\rm Ric}(N)$ is the Ricci curvature of the ambient manifold in the normal direction, we obtain the following (cf. \cite{AMO} when $\Sigma$ is closed):

\begin{quote}
{\bf Theorem \ref{Stability}.}{ \it Let $\Sigma$ be a complete  two sided $H-$surface of finite Abresch-Rosenberg total curvature  in $\hmf$, $H^{2}+\t^{2} \neq 0$. 
\begin{itemize}
\item If $\kappa - 4\t ^{2} > 0$. Then, $\Sigma $ is either an Abresch-Rosenberg $H-$surface, a Hopf cylinder, or
\begin{equation*}
\lambda_{1} < - (4H^{2} + \kappa ).
\end{equation*}
\item If $\k-4\t^{2} < 0$. Then, $\Sigma$ is either an Abresch-Rosenberg $H-$surface, or 
\begin{equation*} 
\lambda_{1} <  - (4H^{2} + \kappa ) - (\kappa-4\t^{2}).
\end{equation*}
\end{itemize}}
\end{quote}

Finally, in Section 6, we apply the Simons' type formula to classify complete $H-$surfaces in $\hmf$ under natural geometric conditions using the Omori-Yau's Maximum Principle. We can summarize Theorem \ref{ap1} and Theorem \ref{ap2} as follows (cf. \cite{Batista} when $\t =0$):

\begin{quote}
{\bf Theorems \ref{ap1} and \ref{ap2}.} {\it Let $\Sigma$  be a complete immersed $H-$surface in $\hmf$, $H^2+\tau^2 \neq 0$. 
\begin{itemize}
\item If  $\k-4\t^{2}>0$, assume that $4(H^{2}+\t^{2}) > \k-4\t^{2}$  and 
$$\sup_{\Sigma} |S| <  \sqrt{2}\sqrt{(H^{2}+\t^{2})+(\kappa -4\t^{2})},$$
where $S$ is the traceless Abresch-Rosenberg shape operator. Then, $\Sigma$ is an Abresch-Rosenberg surface in  $\hmf$.  Moreover, if 
$$\sup_{\Sigma} |S| = \sqrt{2}\sqrt{(H^{2}+\t^{2})+(\kappa -4\t^{2})}$$and there exists one point $p \in \Sigma$ such that  $\abs{S(p)}=\sup_{\Sigma} |S|$, then $\Sigma$ is a Hopf cylinder. 

\item If $\kappa - 4\t^{2} < 0$, assume that $(H^{2}+\t^{2}) > \abs{\k-4\t^{2}}$ and 
$$\sup_{\Sigma} |S| <  -\abs{\alpha} +\sqrt{2(H^{2}+\t^{2}) +\frac{\alpha^{2}}{2}},$$ where $S$ is the traceless Abresch-Rosenberg shape operator. Then, $\Sigma$ is  an Abresch-Rosenberg surface of $\hmf$.

Moreover, if 
$$\sup_{\Sigma} |S| =   -\abs{\alpha} +\sqrt{2(H^{2}+\t^{2}) +\frac{\alpha^{2}}{2}}$$and there exists one point $p \in \Sigma$ such that  $\abs{S(p)}=\sup_{\Sigma} |S|$, then $\Sigma$ is a Hopf cylinder.
\end{itemize}}
\end{quote}

\section{Preliminaries}

Here, we mainly follow \cite{AEG3,Batista,D,ER,ER2,Mi,Scott}. 

\subsection{Homogeneous Riemanniann Manifolds $\mathbb{E}(\kappa ,\tau)$.}

The simply connected homogeneous manifold $\hmf $ is a Riemannian submersion $\pi: \hmf \rightarrow \M^{2}(\kappa )$ over a simply connected surface of constant curvature $\kappa$. The \textbf{fibers}, i.e. the inverse image of a point at $\M^{2}(\kappa)$ by $\pi$, are the trajectories of a unitary Killing field $\xi$, called the \textbf{vertical vector field}.

Denote by $\overline{\nabla}$ the Levi-Civita connection of $\hmf$, then for all $X \in \mx (\hmf )$, the following equation holds (see \cite{Scott}):
$$\overline{\nabla}_{X}\xi = \t  X \wedge \xi ,$$where $\t$ is the \textbf{bundle curvature}. Note that $\t = 0$ implies that $\hmf$ is a product space. Denote by $\overline{R}$ the Riemann curvature tensor of $\hmf$, then, 

\begin{lema}[\cite{D}]
Let $\hmf$ be a homogeneous space with unit Killing field $\xi$. For all vector fields $X,Y,Z,W \in \mx(\hmf)$, we have:

\begin{equation}\label{Cur}
\begin{split}
\langle \overline{R}(X,Y)Z,W \rangle &= (\kappa  - 3\t^{2})\left( \langle X,Z \rangle Y - \langle Y,Z \rangle X \right) \\
    & \quad + (\kappa  - 4\t^{2})\left( \langle \xi, Y \rangle \langle \xi, Z \rangle X - \langle \xi, X \rangle \langle \xi, Z \rangle Y \right) \\
 & \quad - (\kappa  - 4\t^{2})\left(\langle Z,Y \rangle \langle \xi, X \rangle \xi + \langle Z, X \rangle \langle \xi, Y \rangle \xi \right).
\end{split}
\end{equation}
\end{lema}

\subsection{Immersed surfaces in $\hmf$.}

Let $\Sigma \subset \hmf$ be an oriented immersed connected surface. We endow $\Sigma$ with the induced metric of $\hmf$, called the \textbf{first fundamental form}, which we still denote by $\meta{}{}$. Denote by $\nabla$ and $R$ the Levi-Civita connection and the Riemann curvature tensor of $\Sigma$ respectively. Also, denote by $A$ the {\bf shape operator}, $$AX = -\overline{\nabla}_{X}N  \text{ for all }X \in \mx(\Sigma), $$where $N$ is the unit normal vector field along the surface $\Sigma$. Then $II(X,Y) = \langle AX,Y \rangle$ is the {\bf second fundamental form} of $\Sigma$. 

Moreover, denote by $J$ the \textbf{oriented rotation of angle $\frac{\pi}{2}$} on $T\Sigma$, 
$$JX = N \wedge X \text{ for all } X \in \mx(\Sigma) .$$ 

Set $\nu = \langle N,\xi\rangle$ and $\mt = \xi - \nu N$, then, $\nu$ is the normal component of the vertical vector field $\xi$, called the \textbf{angle function}, and $\mt$ is a vector field in $\mathfrak{X}(\Sigma)$ called the \textbf{tangent component of the vertical vector field} $\xi$.

Note that, in a product space $\M^{2}(\kappa) \times \R$, we have a natural projection onto the fiber $\sigma:\Sigma \rightarrow \R$, hence we can define the restriction of $\sigma$ to the surface $\Sigma$, that is, $h:\Sigma \rightarrow \R$, $h = \sigma|_{\Sigma}$. The function $h$ is called the \textbf{height function} of $\Sigma$. So, in $\M^{2}(\kappa) \times \R$, one can easily observes that $\overline{\nabla}\sigma  = \xi$ and hence, $\mt$ is the projection of $\overline{\nabla}\sigma$ onto the tangent plane, $\nabla h = \mt$.

\begin{lema}[\cite{D}]
Let $\Sigma \subset \hmf$ be an immersed surface with unit normal vector field $N$ and shape operator $A$. Let $\mt$ and $\nu$ be the tangent component of the vertical vector field and the angle function respectively. Then, given $X,Y \in \mx(\Sigma)$, the following equations hold:

\begin{eqnarray}
K & = & K_{e} + \t^{2} + (\kappa - 4\t^{2})\nu^{2} ,\label{Gauss}\\
T_{S}(X,Y) & = & (\kappa - 4\t^{2})\nu(\la Y,\mt \ra X - \la X, \mt \ra Y ), \label{Codazzi}\\
\nabla_{X}\mt & = & \nu(AX - \t JX) ,\label{NablaT}\\
d\nu(X) & = & \la \t JX - AX, \mt \ra ,\label{DerNu}\\
\|\mt \|^{2} & + & \nu^{2} =  1 \label{Norm} ,
\end{eqnarray}where $K$ denotes the Gaussian curvature of $\Sigma$, $K_{e}$ denotes the extrinsic curvature and $T_{S}$ is the tensor given by:
$$T_{S}(X,Y) = \nabla_{X}AY - \nabla_{Y}AX - A([X,Y]), \,\, X,Y \in \mx(\Sigma) .$$
\end{lema}

\subsection{$H-$Surfaces in $\hmf$ and the Abresch-Rosenberg differential.}

Let $\Sigma$ be an orientable complete connected $H-$surface immersed in $\hmf$. In terms of a local conformal parameter $z$, the first fundamental form $I = \la,\ra$ and the second fundamental form are given by 

\begin{eqnarray*}
I & = & 2\lambda |dz|^{2}, \\
II & = & Q dz^{2} + 2 \lambda H |dz|^{2} + \overline{Q}d\overline{z}^{2},
\end{eqnarray*}where $Q dz^{2} = - \la \overline{\nabla}_{\partial_{z}}N,\partial_{z} \ra dz^2$ is the usual \textbf{Hopf differential} of $\Sigma$. Hence, in this conformal coordinate, the above Lemma reads as:

\begin{lema}[\cite{ER,ER2,FM}]
\label{EQCP}
Given an immersed $H-$surface $\Sigma \subset \hmf$, the following equations are satisfied: 
\begin{eqnarray}
K & = & K_{e} + \t^{2} + (\kappa - 4\t^{2})\nu^{2} \label{GaussZ}\\
Q_{\overline{z}} & = & \lambda(\kappa - 4\t^{2})\nu \nt \label{CodazziZ}\\
\nt_{z} & = & \frac{\lambda_{z}}{\lambda}\nt + Q\nu \label{NablaTZ}\\
\nt_{\overline{z}} & = & \lambda(H + i \t) \nu \label{NablaTZb}\\
\nu_{z} & = & -(H - i\t)\nt -\frac{Q}{\lambda}\overline{\nt}  \label{DerNuZ} \\
|\nt |^{2} & = & \frac{1}{2} \lambda(1-\nu^{2}) \label{normZ},
\end{eqnarray}where $\nt = \la \mt, \partial_{z} \ra$, $\overline{\nt} = \la \mt, \partial_{\overline{z}} \ra$, $K_{e}$ is the extrinsic curvature and $K$ is the Gaussian curvature.
\end{lema}

For an immersed $H-$surface $\Sigma \subset \hmf$, there is a globally defined quadratic differential, called the {\bf Abresch-Rosenberg differential}.

\begin{definition}[\cite{Rosen,Rosen1}]\label{DefAR}
Given a local conformal parameter $z$ for $I$, the \textbf{Abresch-Rosenberg differential} is defined by:
$$\mathcal{Q}^{AR} = Q^{AR}dz^{2} = (2(H+i \t)Q - (\k - 4\t^{2})\nt^{2})dz^{2},$$moreover, associated to the Abresch-Rosenberg differential we define the \textbf{Abresch-Rosenberg map} $q^{AR}: \Sigma \rightarrow [0,+\infty)$ by:
$$q^{AR} = \frac{ |Q^{AR}|^{2} }{4\lambda^{2}}.$$

Note that $\mathcal{Q}^{AR}$ and $q^{AR}$ do not depend on the conformal parameter $z$, hence $\mathcal{Q}^{AR}$ and $q^{AR}$ are globally defined on $\Sigma$.
\end{definition}

Then using Lemma \ref{EQCP}, we can show

\begin{theorem}[\cite{Rosen,Rosen1}]
\label{thQH}
Let $\Sigma$ be a $H$- surface in $\hmf$, then the Abresch-Rosenberg differential $\mathcal{Q}^{AR}$ is holomorphic for the conformal structure induced by the first fundamental form $I$.
\end{theorem} 

\subsection{Codazzi Pairs on Surfaces}

We shall denote by $\Sigma$ an orientable (and oriented) smooth surface. 

\begin{definition}
A fundamental pair on $\Sigma$ is a pair of real quadratic forms
$(I,II)$ on $\Sigma$, where $I$ is a Riemannian metric.
\end{definition}

Associated with a fundamental pair $(I,II)$ we define the shape operator $S$ of the pair as:
\begin{equation}\label{ii}
 II(X,Y)=I(S(X),Y) \text{ for any } X,Y \in \mx (\Sigma).
\end{equation}

Conversely, it is clear from (\ref{ii})  that the quadratic form $II$ is totally
determined by $I$ and $S$. In other words, a fundamental pair on $\Sigma$ is equivalent to a Riemannian metric on $\Sigma$ together with a self-adjoint endomorphism $S$.

We define the {\bf mean curvature}, the {\bf extrinsic curvature} and the {\bf principal curvatures} of the fundamental pair $(I,II)$ as one half of the trace, the determinant and the eigenvalues of the endomorphism $S$, respectively.

In particular, given local parameters $(x,y)$ on $\Sigma$ such that
$$
I=E\,dx^2+2F\,dxdy+G\,dy^2,\qquad II=e\,dx^2+2f\,dxdy+g\,dy^2,
$$
the mean curvature and the extrinsic curvature of the pair are
given, respectively, by
$$
H(I,II)=\frac{E g+G e-2F f}{2(EG-F^2)},\qquad K_e(I,II)=\frac{eg-f^2}{EG-F^2},
$$moreover, the principal curvatures of the pair are $H(I,II)\pm\sqrt{H(I,II)^2-K_e (I,II)}$.

We shall say that the fundamental pair $(I,II)$ is {\bf umbilical} at $p\in \Sigma$, if $II$ is proportional to $I$ at $p$, or equivalently:
\begin{itemize}
\item if both principal curvatures coincide at $p$, or
\item if $S$ is proportional to the identity map on the tangent plane at $p$, or
\item if $H(I,II)^2-K_e (I,II)=0$ at $p$.
\end{itemize}

We define the {\bf Hopf differential} of the fundamental pair $(I,II)$ as the
(2,0)-part of $II$ for the Riemannian metric $I$. In other words, if we consider $\Sigma$ as a Riemann surface with respect to the metric $I$ and take a local conformal parameter $z$, then we can write
\begin{equation}\label{parholomorfo}
\begin{array}{c} 
I=2\lambda\,|dz|^2, \\[2mm]
II=Q\,dz^2+2\lambda\,H\,|dz|^2+\overline{Q}\,d\zb^2.
\end{array}
\end{equation}

The quadratic form $Q\,dz^2$, which does not depend on the chosen parameter $z$, is known as the {\it Hopf differential} of the pair $(I,II)$. We note that $(I,II)$ is umbilical at $p\in\Sigma$ if and only if $Q(p)=0$.

\begin{remark}
All the above definitions can be understood as natural extensions of the
corresponding ones for isometric immersions of a Riemannian surface in a $3-$dimensional ambient space, where $I$ plays the role of the induced metric and $II$ the role of its second fundamental form.
\end{remark}

A specially interesting case happens when the fundamental pair satisfies the Codazzi equation, that is,

\begin{definition}
We say that a fundamental pair $(I,II)$, with shape operator $S$, is a
{\bf Codazzi pair} if
\begin{equation}\label{ecuacioncodazzi}
\nabla_XSY-\nabla_YSX-S[X,Y]=0,\qquad X,Y\in\mx(\Sigma),
\end{equation}
where $\nabla$ stands for the Levi-Civita connection associated to the Riemannian metric $I$ and $\mx (\Sigma)$ is the set of smooth vector fields on $\Sigma$.
\end{definition}

Let us also observe that, from (\ref{parholomorfo}) and (\ref{ecuacioncodazzi}), a fundamental pair $(I,II)$ is a Codazzi pair if and only if
$$
Q_{\zb}=\lambda\,H_z.
$$

Thus, one has:
\begin{lema}[\cite{Mi}]\label{l0.1}
Let $(I,II)$ be a fundamental pair. Then, any two of the conditions {\rm (i)}, {\rm (ii)}, {\rm (iii)} imply the third:
\begin{itemize}
\item[{\rm (i)}] $(I,II)$ is a Codazzi pair.
\item[{\rm (ii)}] $H$ is constant.
\item[{\rm (iii)}] The Hopf differential is holomorphic.
\end{itemize}
\end{lema}

\begin{remark}
We observe that Hopf's Theorem \cite[p. 138]{Ho}, on the uniqueness of round spheres among immersed constant mean curvature spheres in Euclidean 3-space, can be easily obtained from this result.
\end{remark}

\section{Abresch-Rosenberg Differential and Codazzi Pairs}

One of the main points in the work of Abresch-Rosenberg \cite{Rosen,Rosen1} is to prove that the quadratic differential $\mathcal{Q}^{AR}$ given in Definition \ref{DefAR} is holomorphic on any $H-$surface $\Sigma$ immersed in  $\hmf$. Hence, it is easy to see that such quadratic differential must vanish on a topological sphere by the Poincar\'{e}-Hopf Index Theorem. Then, the authors showed that if the Abresch-Rosenberg differential vanishes on a $H-$surface then it must be invariant under certain one parameter subgroup of isometries of the ambient manifold $\hmf$. In particular, when the surface is a topological sphere, it must be rotationally symmetric. Specifically, they proved:

\begin{theorem}[\cite{Rosen,Rosen1}]
The only topological spheres in $\hmf$ with constant mean curvature $H$ are the rotational invariant spheres.
\end{theorem}  

Lemma \ref{l0.1} tells us that the existence of a holomorphic quadratic differential should imply the existence of a Codazzi pair on any $H-$surface in $\hmf$. When $\hmf$ is a product manifold, i.e. $\t =0$, such Codazzi pair was found a long time ago (cf. \cite{AEG3}). Our goal here is to obtain a Codazzi pair on any $H-$surface such that the Abresch-Rosenberg differential appears as its Hopf differential. First, we recover the case $\t =0$ since it will enlighten the case $\t \neq 0$. 

\subsection{$H-$surfaces in $\hmf$ with $\t =0$.}

Consider a complete immersed $H-$surface $\Sigma \subset \mathbb{M}^{2}(\kappa) \times \mathbb{R}$. According to the notation introduced above, we define the self-adjoint endomorphism $S$ on $\Sigma$ as
\begin{equation}
SX = 2H \, AX - \kappa \langle X,\mt \rangle \mt + \frac{\kappa}{2}\norm{\mt}^{2}X - 2H^{2}X,
\label{Oper}
\end{equation}where $X \in \mathfrak{X}(\Sigma)$, $A$ is the Weingarten operator associated to the second fundamental form, $\mt$ is the tangential component of the vertical vector field $\partial_{t}$ defined in $\mathbb{M}^{2}(\kappa) \times \mathbb{R}$. 

Consider the quadratic form $II_S$ associated to $S$  given by \eqref{Oper}. In \cite{AEG3}, it was shown that $(I,II_S)$ is a Codazzi pair on $\Sigma$ if $H$ is constant. Moreover, it is traceless, i.e., ${\rm tr}(S) = 0= H(I,II_S)$, and the Hopf differential associated to $(I,II_S)$ is the Abresch-Rosenberg differential $\mathcal Q^{AR}$ in $\mathbb{M}^{2}(\kappa) \times \R$.

\subsection{$H-$surfaces in $\hmf$ with $\t \neq 0$.}

The main point in this section is to show that the Abresch-Rosenberg differential has an interpretation in terms of a Codazzi pair defined on any $H-$surface in $\hmf$ when $\t \neq 0$. In this case, we have that $H^{2} + \t^{2} > 0$. Define $\te \in [0,2\pi)$ by
$$e^{2i\te} = \frac{H - i\t}{\sqrt{H^{2} + \t^{2}}}.$$

Let $\Sigma \subset \hmf$ be a $H-$surface and $z$ be a local conformal parameter. Then, up to the complex constant $H + i\t$, we can re-define the Abresch-Rosenberg differential as:

$$Q^{AR}dz^{2} = \bigg( Q -\frac{\k - 4\t^{2}}{2(H + i\t)}\nt^{2} \bigg) dz^{2}. $$ 

One can re-write the above differential as:

$$Q^{AR}dz^{2} = \bigg( Q -\frac{\k - 4\t^{2}}{2\sqrt{H^{2} + \t^{2}}}(e^{i\te}\nt)^{2} \bigg) dz^{2}. $$

Given the tangential vector field $\mt$, define $\mt_{\te} = \cos\te \mt + \sin\te J\mt$, then $\la \mt_{\te}, \partial_{z} \ra = e^{i \te} \nt $, hence, 
$$Q^{AR}dz^{2} = \bigg( \la A\partial_{z},\partial_{z}\ra -\alpha \la \mt_{\te},\partial_{z} \ra^{2} \bigg) dz^{2},$$where $\alpha = \dfrac{\k - 4\t^{2}}{2 \sqrt{H^{2} + \t^{2}}}$  and $A$ is the usual shape operator. This leads us to the following definition:

\begin{definition}\label{ARForm}
Given a $H-$surface $\Sigma \subset \hmf$, the {\bf Abresch-Rosenberg fundamental form} is defined by:
\begin{equation}\label{form}
II_{AR}(X,Y) = II(X,Y) - \alpha \la \mt_{\te},X \ra \la \mt_{\te},Y \ra + \frac{\alpha \norm{\mt}^{2}}{2}\la X,Y \ra,
\end{equation}or equivalently, the {\bf Abresch-Rosenberg shape operator} $S_{AR}$ is defined by:
\begin{equation}\label{ARoperator}
S_{AR} X = A(X) - \alpha \langle \mt_{\theta}, X \rangle \mt_{\theta} + \frac{\alpha \norm{\mt}^{2}}{2}X ,
\end{equation}or, the {\bf traceless Abresch-Rosenberg shape operator} $S$ is defined by:
\begin{equation}\label{ARTraceless}
SX = S_{AR} X - HX = A(X) - \alpha \langle \mt_{\theta}, X \rangle \mt_{\theta} + \frac{\alpha \norm{\mt}^{2}}{2}X - HX,
\end{equation}where $ X,Y \in \mx (\Sigma).$
\end{definition}

First, we examine the geometric properties of the above quadratic form and its relation to the Abresch-Rosenberg differential:

\begin{prop}\label{PropPair}
The following equations hold for the fundamental pair $(I,II_{AR})$: 
\begin{enumerate}
\item[1.]$II_{AR}(\partial_{z},\partial_{z}) dz^{2} = Q^{AR}dz^{2}$, where $z$ is a local conformal parameter for $I$. 
\item[2.] $H(I,II_{AR}) = H(I,II)$.
\item[3.] $K_e(I,II_{AR}) = K_{e}(I,II) +  \alpha \meta{S \mt_{\theta}}{\mt _{\theta}} + \dfrac{\alpha^{2}\norm{\mt}^{4}}{4}$. 
\end{enumerate}

Moreover, the squared norm of the shape operator $\abs{A}^{2}$ and the squared norm of the traceless Abresch-Rosenberg shape operator $\abs{S}^{2}$ satisfy
\begin{equation}\label{ASTraceless}
|A|^{2} = |S|^{2} + 2\alpha \la S\mt_{\theta},\mt_{\theta} \ra + \frac{\alpha^{2}}{2}|\mt|^{4} + 2H^{2}.
\end{equation}

Moreover, it holds
\begin{equation}\label{Eq:Part11}
\frac{|\mt|^{4}}{2} - \frac{\la S\mt_{\theta},\mt_{\theta}\ra^{2}}{|S|^{2}} =\frac{\la S\mt_{\theta}, J \mt_{\theta}\ra^{2}}{|S|^{2}} .
\end{equation}
\end{prop}
\begin{proof}
Consider a local conformal parameter $z$ for $I$. A straightforward computation shows  
$$ II_{AR}(\partial_{z},\partial_{z}) dz^{2} = Q^{AR}dz^{2}$$ and $$ H(I,II_{AR}) = H(I,II).$$ 

Leu us compute $K_{e}(I,II_{AR})$. It is clear that $\norm{\mt_{\theta}} = \norm{J\mt_{\theta}} = \norm{\mt}$, then; 
\begin{equation}\label{CJ1}
\begin{split}
II_{AR}(\mt_{\theta},\mt_{\theta}) &= II(\mt_{\theta},\mt_{\theta}) - \frac{\alpha \norm{\mt}^{4}}{2}.\\
II_{AR}(\mt_{\theta},J\mt_{\theta}) &= II(\mt_{\theta},J\mt_{\theta}).\\
II_{AR}(J\mt_{\theta},J\mt_{\theta}) &= II(J\mt_{\theta},J\mt_{\theta}) + \frac{\alpha \norm{\mt}^{4}}{2}.
\end{split}
\end{equation} 

From the definition of the Abresch-Rosenberg quadratic form, we have $K_{e}(I,II_{AR}) = K_{e}(I,II)$ on the set $\mathcal{U} = \{ p\in \Sigma \, :  \, \norm{\mt}^{2}(p) = 0 \} $. Then, take $p \in \Sigma \setminus \mathcal{U}$ and consider the orthonormal basis in $T_{p}\Sigma$ defined by:
\begin{equation*}\label{CJ2}
e_{1} = \frac{\mt_{\theta}}{\norm{\mt}} \text{ and } e_{2} = \frac{J \mt_{\theta}}{\norm{\mt}}.
\end{equation*}

From \eqref{CJ1}, we obtain:
\begin{equation}\label{SFO1}
II_{AR}(e_{1},e_{1}) - II _{AR}(e_{2},e_{2}) = II(e_{1},e_{1}) - II(e_{2},e_{2}) - \alpha\norm{\mt}^{2},
\end{equation}and, since $\{e_{1},e_{2}\}$ is orthonormal at $p$ and \eqref{CJ1}, we have
\begin{eqnarray*}
K_{e}(I,II_{AR}) & = & II_{AR}(e_{1},e_{1})II_{AR}(e_{2},e_{2}) -II_{AR}(e_{1},e_{2})^2\\[3mm]
             & = & II(e_{1},e_{1})II(e_{2},e_{2}) -II(e_{1},e_{2})^2 \\
 & & \qquad + \frac{\alpha}{2}(II(e_{1},e_{1}) -II(e_{2},e_{2}))\norm{\mt}^{2} -\frac{\alpha^{2}}{4} \norm{\mt }^{4} .
\end{eqnarray*}

On the one hand, substituting \eqref{SFO1} into the above formula of $K_{e}(I,II_{AR})$ and simplifying terms, we get at $p$:
\begin{equation}\label{Kee}
K_{e}(I,II_{AR}) = K_{e}(I,II) + \frac{\alpha}{2}\left( II_{AR}(\mt _\theta, \mt_\theta) - II_{AR}(J\mt _\theta, J\mt_\theta)\right) + \frac{\alpha^{2}\norm{\mt}^{4}}{4}.
\end{equation}

On the other hand, recall that $S$ is traceless and hence, at a point $p\in \Sigma$, we can consider an orthonormal basis $\set{E_{1}, E_{2}}$ of principal directions for $S$, i.e, 
$$ S E_{1} = \lambda E_{1} , \,  SE_{2} = -\lambda E_{2} \text{ and } |S|^{2} =2 \lambda ^{2}.$$

Then, there exists $\beta \in [0,2\pi)$ such that 
$$ \mt_{\theta} = |\mt| ( \cos \beta \, E_{1} +  \sin \beta E_{2}) ,$$and hence, one can easily check
\begin{equation}\label{RRR}
\meta{S \mt _{\theta}}{\mt _{\theta}} = - \meta{S J\mt _{\theta}}{ J \mt_{\theta}}  .
\end{equation} 

Hence, from \eqref{ARTraceless} and \eqref{RRR}, we have 
\begin{equation*}
\begin{split}
II_{AR}(\mt _\theta, \mt_\theta) - II_{AR}(J\mt _\theta, J\mt_\theta) &= \meta{S \mt _{\theta}}{\mt _{\theta}} - \meta{S J\mt _{\theta}}{ J \mt_{\theta}} \\[3mm] 
 &= 2 \meta{S \mt_{\theta}}{\mt _{\theta}},
\end{split}
\end{equation*}thus, substituting the last equation into \eqref{Kee} yields the expression for $K_{e}(I,II_{AR})$. Also, a straightforward computation using \eqref{RRR} gives that
\begin{equation*}
\frac{|\mt|^{4}}{2} - \frac{\la S\mt_{\theta},\mt_{\theta}\ra^{2}}{|S|^{2}} =\frac{\la S\mt_{\theta}, J \mt_{\theta}\ra^{2}}{|S|^{2}} ,
\end{equation*}which shows \eqref{Eq:Part11}.

Finally, \eqref{ASTraceless} can be easily obtained by observing that $\abs{A}^{2} = 4H^{2} - 2 K_{e}$ and $\abs{S}^{2} =2 q^{AR} =2( H^{2}-K_{e}(I,II_{AR}) )$.
\end{proof}

Hence,  Lemma \ref{l0.1} implies:

\begin{lema}
Given any $H-$surface in $\hmf$, $H^2 +\t^2 \neq 0$, it holds: 
\begin{quote}
$\mathcal Q^{AR}$ is holomorphic if and only if $(I,II_{AR})$ is a Codazzi pair.
\end{quote}
\end{lema}

So, we can summarize the above discussion in the following

\begin{lema}[Key Lemma]\label{Lem:Key}
Given a $H-$surface $\Sigma \subset \hmf$, $H^2 +\t^2 \neq 0$, consider the symmetric $(2,0)-$tensor given by 
$$II_{AR}(X,Y) = II(X,Y) - \alpha \la \mt_{\te},X \ra \la \mt_{\te},Y \ra + \frac{\alpha \norm{\mt}^{2}}{2}\la X,Y \ra, $$where 
\begin{itemize}
\item $\alpha = \dfrac{\k - 4\t^{2}}{2 \sqrt{H^{2} + \t^{2}}}$,
\item $e^{2i\te} = \frac{H - i\t}{\sqrt{H^{2} + \t^{2}}}$ and 
\item $\mt_{\te} = \cos\te \mt + \sin\te J\mt$.
\end{itemize}

Then, $(I, II_{AR})$ is a Codazzi pair with constant mean curvature $H$. Moreover, the $(2,0)-$part of $II_{AR}$ with respect to the conformal structure given by $I$ agrees (up to a constant) with the Abresch-Rosenberg differential.
\end{lema}

And, as a consequence of Lemma \ref{l0.1} we have:

\begin{corollary}
\label{BoundedC}
Let $\Sigma$ be a $H-$surface in $\hmf$, $H^{2}+\t^{2} \neq 0$. Then, the following conditions are equivalent
\begin{itemize}
\item $\abs{S}$ is bounded, 
\item $\abs{A}$ is bounded, 
\item $\abs{K}$ is bounded.
\end{itemize}
\end{corollary}
\begin{proof}
On the one hand, from \eqref{ASTraceless} is clear that $\abs{A}$ is bounded if, and only if, $\abs{S}$ is bounded. On the other hand, from $4H^{2}-2K_{e}=\abs{A}^{2}$ and the Gauss equation, we have that $\abs{K}$ is bounded if, and only if, $\abs{A}$ is bounded.
\end{proof}


\section{Simons' type formula in $\hmf.$}

In this section, we will obtain a Simons' type formula for the traceless Abresch-Rosenberg shape operator $S$ defined on a $H-$surface  $\Sigma \subset \hmf$, $\t \neq 0$. The Simons' type formula follows form the fact that the traceless Abresch-Rosenberg shape operator is Codazzi and the work of Cheng-Yau \cite{Yau}.

\begin{theorem}\label{Simfor}
Let $\Sigma$ be a $H-$surface in $\hmf$. Then, the traceless Abresch-Rosenberg shape operator satisfies
\begin{equation}\label{lapla}
\frac{1}{2} \Delta |S|^{2} =|\n S|^{2}  + 2 K |S|^{2},
\end{equation}or, equivalently, away from the zeroes of $\abs{S}$,  
\begin{equation}\label{lapla2}
\abs{S} \, \Delta \abs{S} - 2 K \abs{S}^2 =  \abs{\nabla \abs{S}}^{2}.
\end{equation}
\end{theorem}
\begin{proof}
Since $(I,II_{AR})$ is a Codazzi pair on $\Sigma$, from Lemma \ref{l0.1} we get that $(I,II_{AR}- H\, I)$ is also a Codazzi pair on $\Sigma$, observe that $II_{AR}- H\, I$ is nothing but the traceless Abresch-Rosenberg fundamental form. Hence, from \eqref{Simmons} (cf. \cite{Yau}), we obtain
$$ \frac{1}{2}\Delta \abs{S}^{2} = \abs{\nabla S} ^{2} + 2 K \abs{S}^{2} ,$$as claimed. Here, we have used that $R_{ijij}= K $ since we are working in dimension two. Also, since $S$ is traceless, the two eigenvalues, $\lambda _{i}$, are opposite signs so $(\lambda _{i} -\lambda _{j})^{2} = 4 \lambda _{i}^{2} = 2 \abs{S}^2$.

To obtain \eqref{lapla2}, using $\Delta \abs{S}^{2} = 2 \abs{S}\Delta \abs{S} + 2 \abs{\nabla \abs{S}}^{2}$, we get
\begin{equation*}
\abs{S}\Delta \abs{S} + \abs{\nabla \abs{S}}^{2} = |\n S|^{2} +2 K \abs{S}^{2}
\end{equation*}

Thus, since $\Sigma$ has dimension two and $S$ is traceless and Codazzi, it holds (cf. \cite{Brendle})
$$ |\n S|^{2} = 2\abs{\nabla \abs{S}}^{2} , $$and we finally obtain \eqref{lapla2}.
\end{proof}

\subsection{Classification results for $H-$surfaces in $\hmf$}

First, we shall recall the classification theorem when $q^{AR}$ vanishes identically.

\begin{lema}[\cite{Rosen,Rosen1,ER2}]
\label{QVa}
Let $\Sigma \subset \hmf$ be a complete $H-$surface whose Abresch-Rosenberg differential vanishes. Then $\Sigma$ is invariant by certain one parameter subgroup of isometries of $\hmf$.\\ 
\end{lema}

Hence, the above Lemma motivates the following:

\begin{definition}\label{Def:ARSurface}
Let $\Sigma$ be a complete $H-$surface in $\hmf$. We say that $\Sigma$ is an {\bf Abresch-Rosenberg surface} if its Abresch-Rosenberg differential vanishes identically.
\end{definition}

Since $\hmf$ is a Riemannian submersion $\pi :\hmf \to \M^{2}(\kappa)$, given $\gamma$ a regular curve  in $\M^{2}(\k)$, $\Sigma _\gamma:=\pi^{-1}(\gamma)$ is a surface in $\hmf$ satisfying that $\xi$ is a tangential vector field along $\Sigma _\gamma$, in this case $\nu = 0$. So, $\xi$  is a parallel vector field along $\Sigma _\gamma$ and hence $\Sigma _\gamma$ is flat and its mean curvature is given by $2 H = k_{g}$, where $k_{g}$ is the geodesic curvature of $\gamma$ in $\M^{2}(\kappa)$ (cf. \cite[Proposition 2.10]{EO}). We will call $\Sigma _\gamma:=\pi^{-1}(\gamma)$ a {\bf Hopf cylinder} in $\hmf$ over the curve $\gamma$. If $\gamma$ is a closed curve, $\Sigma _\gamma$ is a flat Hopf cylinder and additionally, if $\pi$ is a circle Riemannian submersion, $\Sigma _\gamma$ is a {\bf Hopf torus}. The latter case occurs when $\hmf$ is a Berger sphere (see, \cite[Theorem 1]{To2}), i.e., $\kappa =1$ and $\tau \neq 0$. 

Now, with the definition of Hopf cylinders in hand, we classify $H-$surfaces in $\hmf$ with constant non-zero Abresch-Rosenberg map $q^{AR}$ (cf. \cite{ER2} for a different proof).

\begin{theorem}\label{Lem:qconstant}
Let $\Sigma \subset \hmf$ be a complete $H-$surface and suppose $q^{AR}$ is a positive constant map on $\Sigma$, then $\Sigma$ is a Hopf cylinder over a complete curve of curvature $2H$ on $\mathbb{M}^{2}(\k)$.
\end{theorem}
\begin{proof}
We can assume, without loss of generality, that $\Sigma $ is simply-connected by passing to the universal cover. Since $q^{AR}$ is a positive constant,  (cf. \cite[Main Lemma]{Mi}) we have $ 0= \Delta \ln q ^{AR} = 4 K$, that is, the Gaussian curvature vanishes identically on $\Sigma$. Moreover, since $ q^{AR} = H^2 - K _e (I, II_{AR}) = c^{2} >0  $ is constant, we obtain that $K_e (I, II_{AR}) = H^2 - c^{2} $ is constant on $\Sigma $. 

On the one hand, since $\Sigma$ is simply connected, $q^{AR}= c^{2}>0$, there exists a global conformal parameter $z= x+i y$ (cf. \cite[Main Lemma]{Mi}), so that
$$ cI = dx^{2} + dy^{2} \text{ and } cII_{AR} = (H+c)dx^{2} + (H-c) dy^{2}.$$

On the other hand, combining the Gauss equation \eqref{GaussZ} and the expression of $K_{e}(I,II_{AR})$ given by item 3 in Proposition \ref{PropPair}, we obtain
$$ \tau ^{2} + (\kappa -4 \t ^{2})(1-\abs{\mt}^{2}) = \alpha \meta{S\mt _{\theta}}{\mt _{\theta}} +\frac{\alpha ^{2}\abs{\mt}^{4}}{4} - H^{2}+c^{2} ,$$or, in other words, 
$$  \frac{\alpha ^{2}\abs{\mt}^{4}}{4} +(\kappa -4\t ^{2})\abs{\mt}^{2} +  \alpha \meta{S\mt _{\theta}}{\mt _{\theta}} + c^2 - H^2- \tau^{2} - (\kappa - 4 \tau ^2) = 0 \text{ on } \Sigma .$$

Set $u=\meta{\mt _{\theta}}{\partial_{x}}$ and $v=\meta{\mt_{\theta}}{\partial _{y}}$. Thus, the last equations imply that there exists a polynomial in the variables $u$ and $v$, $P(u,v)$, whose leading term is $ \frac{\alpha }{2}(u^{2}+v^{2})^{2} $. Since $\alpha  \neq 0$, we obtain that $\norm{\mt } $ is constant on $\Sigma$ and hence $\nu $ is constant along $\Sigma$, which implies that $\Sigma $ is a vertical cylinder (cf. \cite[Theorem 2.2]{ER}).
\end{proof}

\section{Finite Abresch-Rosenberg Total Curvature}

Complete $H-$surfaces $\Sigma \subset \r ^{3}$ of finite total curvature, i.e., those whose $L^{2}-$norm of its traceless second fundamental form is finite, are of capital importance on the comprehension of $H-$surfaces. If $\Sigma $ has constant nonzero mean curvature and finite total curvature, then it must be compact. If $\Sigma$ is minimal, Osserman's Theorem gives an impressive description of them. 

When we consider complete $H-$surfaces in $\hmf$, the traceless part of the second fundamental form encodes less information about the surface. So, it is also convenient to consider the traceless part of the Abresch-Rosenberg fundamental form.

\begin{definition}
Let $\Sigma \subset \hmf$ be a complete $H-$surface, $H^{2}+\tau ^{2} \neq 0$. We say that $\Sigma$ has finite Abresch-Rosenberg total curvature if the $L^{2}-$norm of the traceless Abresch-Rosenberg form is finite, i.e, 
$$ \int_{\Sigma} \abs{S} ^{2} dv_{g} < +\infty .$$
where $dv_{g}$ is the volume element of $\Sigma$.
\end{definition}

We must point out here that the family of complete constant mean curvature surfaces of finite Abresch-Rosenberg total curvature is large. Obviously, any Abresch-Rosenberg surface has finite Abresch-Rosenberg total curvature. We focus on $H=1/2$ surfaces in $\h ^2 \times \r $ to show this fact. Recall the following result of Fern\'{a}ndez-Mira:

\begin{quote}
{\bf Theorem \cite[Theorem 16]{FM}.} {\it Any holomorphic quadratic differential on an open simply connected Riemann surface is the Abresch-Rosenberg differential of some complete surface $\Sigma$ with $H=1/2$ in $\h ^2 \times \r$. Moreover, the space of noncongruent complete mean curvature one half surfaces in $\h^2 \times \r$ with the same Abresch-Rosenberg differential is generically infinite.}
\end{quote}

Then, we take the disk $\d $ as our open Riemann surface and a holomorphic quadratic differential $\mathcal Q $ on $\d$ that extends continuously to the boundary. Let $\Sigma$ be the $H=1/2$ surface constructed in \cite[Theorem 16]{FM}. Now, we will see that $\Sigma $ has finite Abresch-Rosenberg total curvature. 

For a conformal parameter $z \in \d $, we have $I = 2 \lambda |dz|^2$ and hence $\mathcal Q = f(z) dz^2$, for a holomorphic function $f : \d \to \c $ that extends continuously to the boundary. Then, the squared norm of the traceless Abresch-Rosenberg shape operator is given by $\abs{S}^2 = \frac{\abs{f}^2}{4\lambda ^2}$ and $dv_g = 4 \lambda ^2 \abs{dz}^2$. Thus, we have 
$$ \int _\Sigma \abs{S}^2 \, dv_g =  \int _\d \abs{f(z)}^2 \abs{dz}^2 < +\infty ,$$as claimed. 

The study of constant mean curvature surfaces of finite Abresch-Rosenberg total curvature is complementary to the study of surfaces of finite total curvature (see \cite{HR,HM} and references therein). Note that we are assuming $H^{2}+\tau ^{2}\neq 0$, otherwise we consider the usual Abresch-Rosenberg operator given by \eqref{Oper}. In \cite{BCF}, the authors studied complete $H-$surfaces of finite Abresch-Rosenberg total curvature in product spaces $\mer$. The fundamental tool used in \cite{BCF} is the Simons' type formula for $\abs{S}$ developed in \cite{Batista} when $\t =0$. Hence, using our Simons' type formula (Theorem \ref{Simfor}), we can obtain:

\begin{prop}\label{In}
Let $\Sigma$ be an immersed $H-$surface, $H^{2}+\t^{2} \neq 0$, in $\hmf$ and set $u = \abs{S}$, where $S$ is the traceless Abresch-Rosenberg form. Then 
\begin{equation}
\label{In1}
-\Delta u \leq a u^{3} + bu,
\end{equation}
where $a$, $b$ are constants depending on $\kappa -4\t^{2}$ and $H$.
\end{prop}
\begin{proof}
First, from \eqref{ASTraceless} and $4H^{2}-\abs{A}^{2}=2K_{e}$, we have
\begin{equation}
\label{Eq1}
K_{e} = H^{2} -\frac{1}{2}\abs{S}^{2} - \alpha \la S \mt_{\theta}, \mt_{\theta} \ra  - \alpha^{2} \frac{\abs{\mt}^{4}}{4}.
\end{equation}

Substituting  \eqref{Eq1} into the Gauss equation $K = K_{e} + \t^{2} + (\kappa-4\t^{2})\nu^{2}$, we can rewrite the Gaussian curvature $K$ as follows
\begin{equation}
\label{eq2}
K = (\kappa-4\t^{2})(1-\abs{\mt}^{2})+\t^{2}+H^{2}-\frac{1}{2}\abs{S}^{2}-\alpha \la S\mt_{\theta},\mt_{\theta} \ra - \alpha^{2}\frac{\abs{\mt}^{4}}{4}.
\end{equation}

Next, substituting \eqref{eq2} into \eqref{lapla2}, we obtain:
\begin{equation}
\label{inequality1}
\begin{split}
\Delta \abs{S} & \geq 2\abs{S}\bigg( (\k-4\t^{2})\nu^{2}+\t^{2}+H^{2}-\frac{1}{2}\abs{S}^{2}-\alpha \la S\mt_{\theta},\mt_{\theta} \ra - \alpha^{2}\frac{\abs{\mt}^{4}}{4} \bigg) \\
& \geq -\abs{S}^{3} - \abs{S} \bigg( -2 \min\{0,\k-4\t^{2} \}-2\t^{2}-2H^{2}+ \frac{1}{2}\alpha^{2}+2\alpha \la S \mt_{\theta},\mt_{\theta} \ra  \bigg).
\end{split}
\end{equation}

Since $S$ is traceless, we have that $\abs{S\mt_{\theta}} = \frac{1}{\sqrt{2}}\abs{S}\abs{\mt_{\theta}}$, and using the Schwarz inequality $\abs{\la S\mt_{\theta},\mt_{\theta} \ra} \leq \abs{\mt_{\theta}}\abs{S\mt_{\theta}}$, we see that
\begin{equation}
\label{eq3}
- \abs{S}(2 \alpha \la S\mt_{\theta},\mt_{\theta} \ra) \geq 
-\frac{2}{\sqrt{2}}\abs{\alpha}\abs{S}^{2} \geq -\frac{\abs{\alpha}}{\sqrt{2}}\abs{S}^{3} - \frac{\abs{\alpha}}{\sqrt{2}}\abs{S}.
\end{equation}

Finally, combining \eqref{eq3} with \eqref{inequality1} yields
\begin{equation*}
\Delta \abs{S} \geq -\bigg(1+\frac{\abs{\alpha}}{\sqrt{2}} \bigg)\abs{S}^{3} - \bigg( -2 \min\{0,\k-4\t^{2} \}-2\t^{2}-2H^{2}+ \frac{1}{2}\alpha^{2}+ \frac{\abs{\alpha}}{\sqrt{2}} \bigg) \abs{S},
\end{equation*}which proves \eqref{In1}.
\end{proof}

Any $H-$surface $\Sigma$ in $\hmf$ satisfies a Sobolev type inequality of the form (cf. \cite{HS})
\begin{equation}
\label{SOB}
\norm{f}_{2} \leq C_{0} \norm{\nabla f}_{1} + C_{1} \norm{f}_{1}, \, \, f \in C^{\infty}_{0}(\Sigma), 
\end{equation}where $\norm{f}_{p}$ denotes the $L^{p}(\Sigma)$-norm of $f$ and $C_{0}$, $C_{1}$ are constants that depend only on the mean curvature $H$. Here, $C _{0}^\infty (\Sigma)$, as always, stands for the linear space of compactly supported piecewise smooth functions on $\Sigma$. 

Now, let $p \in \Sigma$ be a fixed point. Consider the intrinsic distance function $d(x,p)$ to $p$ and define the open sets
$$B(R) = \{x \in \Sigma : \, d(p,x) < R \}  \ \ \ and \ \ \ E(R) = \{x \in \Sigma: \, d(x,p) > R \},$$ then, with the above notations, we can show the following: 

\begin{theorem}\label{ZeroInfinity}
Let $\Sigma \subset \hmf$ be a complete $H-$surface, $H^{2}+\tau ^{2} \neq 0$, of finite Abresch-Rosenberg total curvature, that is, 
$$  \int _{\Sigma} \abs{S} ^{2} dv_{g}< +\infty ,$$ then $\abs{S}$ goes to zero uniformly at infinity. More precisely, there exist positive constants $A$, $B$ and a positive radius $R_{\Sigma}$, determined by the condition $B \int_{E(R_{\Sigma})}  \abs{S}^{2} \leq 1$, such that for $u = \abs{S}$ and for all $R \geq R_{\Sigma}$, we have
\begin{equation}
\label{UCZTI}
\|u\|_{\infty,E(2R)} = \sup_{x \in E(2R)}u(x) \leq A \big( \int_{E(R)}|S|^{2} dv_{g} \big)^{\frac{1}{2}}
\end{equation}
and, there exist positive constants $D$ and $E$ such that the inequality $\int_{\Sigma}\abs{S}^{2} dv_{g} \leq D$ implies  
\begin{equation*}
\|u\|_{\infty} = \sup_{x \in \Sigma}u(x) \leq  E \int_{\Sigma}\abs{S}^{2}dv_{g}.
\end{equation*}
\end{theorem}

\begin{proof}
Since the function $u=\abs{S}$ satisfies the Sobolev type inequality  \eqref{SOB} and the inequality \eqref{In1}, we can now work as in the proof of \cite[Theorem 4.1]{BCS} to show that $u$ satisfies the inequality \eqref{UCZTI}, letting $R$ go to infinity we get that $\abs{S}$ goes to zero uniformly at infinity.
\end{proof}

Next, we study $H-$surfaces $\Sigma$ in $\hmf$, $H^{2}+\t^{2} \neq 0$, of finite Abresch-Rosenberg total curvature. Despite what happens in $\r ^3$, a $H-$surface $\Sigma \subset \hmf$ with finite Abresch-Rosenberg total curvature is not necessarily conformally equivalent to a compact surface minus a finite number of points, in particular, $\Sigma$ is not necessarily parabolic. For example, the complete $H=1/2$ surface constructed above is hyperbolic. However, we obtain:

\begin{theorem}\label{ThConfEquiv}
Let $\Sigma$ be a complete surface on $\hmf$, $H^{2}+\t^{2} \neq 0$, of finite Abresch-Rosenberg total curvature. Suppose one of the following conditions holds 
\begin{itemize}
\item[1.] $\kappa-4\t^{2}>0$ and $H^{2}+\t^{2} > \frac{\kappa-4\t^{2}}{4} $.
\item[2.] $\kappa-4\t^{2}<0$ and $H^{2}+\t^{2} > -\frac{(\sqrt{5}+2)}{4}(\kappa-4\t^{2}) $.
\end{itemize}

Then, $\Sigma$  must be compact.
\end{theorem}

\begin{proof}
From \eqref{eq2}, the Gaussian curvature can be written as
\begin{equation*}
K = (\kappa-4\t^{2})(1-\abs{\mt}^{2})+\t^{2}+H^{2}-\frac{1}{2}\abs{S}^{2}-\alpha \la S\mt_{\theta},\mt_{\theta} \ra - \alpha^{2}\frac{\abs{\mt}^{4}}{4}.
\end{equation*}

Now, $\abs{\mt_{\theta}} \leq 1$, $S$ traceless and the Schwarz inequality imply $$-\alpha \la S\mt_{\theta},\mt_{\theta} \ra \geq -\abs{\alpha}\frac{\abs{S}\abs{\mt_{\theta}}}{\sqrt{2}} \geq -\frac{\abs{\alpha}\abs{S}}{\sqrt{2}} ,$$therefore
\begin{equation*}
K \geq (\kappa-4\t^{2})\nu^{2} + (H^{2}+\t^{2}) -\frac{1}{2}\abs{S}^{2} -\abs{\alpha}\frac{\abs{S}}{\sqrt{2}} - \alpha^{2}\frac{\abs{\mt}^{4}}{4}.
\end{equation*}

If $\kappa-4\t^{2} > 0$, then 
$$K \geq  (H^{2}+\t^{2})  - \frac{\alpha^{2}}{4} -\frac{1}{2}\abs{S}^{2} -\alpha\frac{\abs{S}}{\sqrt{2}}.$$

If $\kappa-4\t^{2}<0$, then
$$K \geq  (\kappa-4\t^{2})+(H^{2}+\t^{2})  - \frac{\alpha^{2}}{4} -\frac{1}{2}\abs{S}^{2} + \alpha\frac{\abs{S}}{\sqrt{2}}.$$

In both cases, the hypothesis and the fact that $\abs{S}$ goes to zero uniformly imply that there exists a compact set $\overline{\Omega}$ and $\epsilon >0$ (depending on the compact set) such that the Gaussian curvature satisfies 
$$K  (p)\geq \epsilon >0 \text{ for all } p \in \Sigma \setminus \overline{\Omega} .$$

Therefore, Bonnet Theorem implies that $d(p,\partial \Sigma \setminus \overline{\Omega})$ is uniformly bounded for all $p \in \Sigma \setminus \overline{\Omega}$. Thus, $\Sigma$ must be compact. 

\end{proof}

\begin{remark}
The above result was obtained in \cite{BCF} when $\t =0$. We could have extended other results from \cite{BCF}, but we left this to the interested reader.
\end{remark}

\subsection{First Eigenvalue of a Schr\"{o}dinger operators}

We will use the Simons' type formula \eqref{lapla2} to estimate the first eigenvalue $\lambda_{1}(L)$ of a Schr\"{o}dinger operator $L$ defined on a complete $H-$surface $\Sigma$ in $\hmf$ of finite Abresch-Rosenberg total curvature.  

Set $V \in C^0 (\Sigma) $ and consider the differential linear operator, called {\bf Schr\"{o}dinger operator}, given by 
$$ \begin{matrix} L : & C _{0}^\infty (\Sigma) &\to & C _{0}^\infty (\Sigma) \\
  & f & \to & Lf := \Delta f + V \, f ,\end{matrix}$$where $\Delta $ is the Laplacian with respect to the induced Riemannian metric on $\Sigma$.
  
Given a relatively compact domain $\Omega \subset \Sigma$, it is well-known (cf. \cite{ICha84,DGilNTru83}), that there exists a positive function $\rho : \Sigma \to \r$ such that 
\begin{equation*}\label{FEF}
\left\{ \begin{matrix}
-\Delta\rho  &=& (V + \lambda_{1}(L ,\Omega)) \rho  & \text{ in } & \Omega \\
 \rho & =&  0   & \text{ on } & \partial \Omega,
 \end{matrix}\right.
\end{equation*}where 
$$ \lambda _1 (L, \Omega)= {\rm inf}\set{\frac{\int _\Omega (\norm{\nabla f}^2 - V f^2)dv_g}{\int _\Omega f^2 \, dv_g} \, : \, \, f \in C_{0}^\infty (\Omega)} ,$$that is, $\lambda _1 (L,\Omega)$ and $\rho$ are the {\bf first eigenvalue} and {\bf first eigenfunction}, respectively, associated to the Schr\"{o}dinger operator $L$ on $\Omega \subset \Sigma$.

Now, we can consider the infimum over all the relatively compact domains in $\Sigma$ and we can define the infimum of the spectrum of $L$ as 
$$ \lambda _{1} (L) := {\rm inf}\set{ \lambda _{1}(L,\Omega) \, : \,\, \Omega \subset \Sigma \text{ relatively compact} } ,$$in particular, 
$$  \displaystyle{\lambda _{1} (L)  : = {\rm \inf}_{i \to +\infty} \lambda _{1} (L, \Omega _{i})},$$for any compact exhaustion $\set{\Omega _{i}} $ of $\Sigma$. 

\begin{remark}
It is standard that the regularity conditions above can be relaxed, but this is not important in our arguments. 
\end{remark}

First, we will relate the Simons' type formula to the first eigenvalue of any Schr\"{o}dinger operator. 

\begin{lema}\label{Funtes}
Let $\Sigma \subset \hmf $ be a complete $H-$surface, $H^{2}+\tau ^{2} \neq 0$, and $\Omega \subset \Sigma$ a relatively compact domain. Denote by $ \lambda _{1}(L,\Omega) $ and $\rho_{\Omega} $ the first eigenvalue and first eigenfunction, respectively, associated to the Schr\"{o}dinger operator $L := \Delta + V$ on $\Omega$, $V\in C^0 (\Omega)$. Set  
$$C_{\Omega} = |S|(V + \lambda_{1}(L,\Omega)) + \Delta |S| ,$$where $S$ is the traceless Abresch-Rosenberg shape operator. Given $\phi \in C^{\infty}_{0} (\Omega ')$, $\overline{\Omega '} \subset \Omega$, then 
\begin{equation}\label{Eq:Complete}
\int _{\Omega} \phi ^{2} \abs{S} C_{\Omega} \leq \int _{\Omega} \abs{S}^{2}\norm{\nabla \phi}^{2} .
\end{equation}
\end{lema}
\begin{proof}
Set $\rho_{\Omega} = \rho$ and $\lambda_{1}(L,\Omega) = \lambda_{1}$. By the Maximum Principle $\rho >0$ in $\Omega$. Set $w := \ln \rho $ in $\Omega '$ and note that it is well defined. Moreover, it holds
$$ \Delta w = -\left( V+\lambda _{1}\right) -\abs{\nabla w}^{2} .$$ 

Set $\psi = \phi \abs{S}$. On the one hand, by Stokes' Theorem 
\begin{equation*}
\begin{split}
 0 & =  \int _{\Omega}{\rm div}\left( \psi ^{2}\nabla w\right) \\
  & = \int _{\Omega}\psi ^{2} \Delta w + \int_{\Omega} 2  \psi \meta{\nabla \psi}{\nabla w} \\
  & = -\int _{\Omega}\psi ^{2}\left( V+\lambda _{1}\right) -\int _{\Omega} \psi ^{2}\abs{\nabla w}^{2} + \int_{\Omega} 2 \psi \meta{\nabla \psi}{\nabla w}  \\
  & \leq -\int _{\Omega}\psi ^{2}\left( V+\lambda _{1}\right) + \int _{\Omega} \abs{\nabla \psi}^{2},
\end{split}
\end{equation*}where we have used $ -\psi ^{2}\abs{\nabla w}^{2} + 2 \psi \meta{\nabla \psi}{\nabla w} \leq \abs{\nabla \psi}^{2}$. In other words, we have 
\begin{equation}
\label{inequal}
 \int _{\Omega}\psi ^{2}\left( V+\lambda _{1}\right) \leq  \int _{\Omega} \abs{\nabla \psi}^{2}.
\end{equation}

On the other hand, using the definition of $\psi$ we get
$$\abs{\nabla \psi} ^{2} = \phi ^{2} \abs{\nabla \abs{S}}^{2} +2 \phi \abs{S} \meta{\nabla \phi}{\nabla \abs{S}} + \abs{S}^{2}\abs{\nabla \phi}^{2} ,$$and, since 
$$ \frac{1}{2} {\rm div}\left( \phi ^{2} \nabla \abs{S}^{2}\right) = \phi ^{2}\abs{S}\Delta \abs{S} +\phi ^{2} \abs{\nabla \abs{S}}^{2} +2 \phi \abs{S} \meta{\nabla \phi}{\nabla \abs{S}} ,$$we have that
$$ \abs{\nabla \psi} ^{2} = \frac{1}{2} {\rm div}\left( \phi ^{2} \nabla \abs{S}^{2}\right) -\phi ^{2}\abs{S}\Delta \abs{S} + \abs{S}^{2}\abs{\nabla \phi}^{2}, $$so, taking integrals and using Stokes' Theorem, we obtain
\begin{equation}
\label{equal}
 \int _{\Omega} \abs{\nabla \psi}^{2} = -\int _{\Omega} \phi ^{2}\abs{S}\Delta \abs{S}+ \int_{\Omega} \abs{S}^{2}\abs{\nabla \phi}^{2}.
\end{equation}

Thus, combining \eqref{equal} with \eqref{inequal} we get \eqref{Eq:Complete}.
\end{proof}

Now, we can use Lemma \ref{Funtes} to estimate the first eigenvalue $\lambda_{1}(L)$ of the Schr\"{o}dinger operator $L = \Delta + V$. Specifically,

\begin{theorem}\label{Thm:Estimate1}
Let $\Sigma$ be a complete two-sided $H-$surface in $\hmf$ of finite Abresch-Rosenberg total curvature and $H^{2} + \t^{2} \neq 0$. Denote by $\lambda_{1}(L)$ the first eigenvalue associated to the Schr\"{o}dinger operator $L := \Delta + V$, $V\in C^0 (\Sigma)$. Then, $\Sigma$ is either an Abresch-Rosenberg surface, a Hopf cylinder or 
\begin{equation}
\label{Esl}
\lambda _{1}(L) < - {\rm inf}_{\Sigma}\set{V+2K}.
\end{equation}
\end{theorem}
\begin{proof}
Assume that $\abs{S}$ is not identically constant on $\Sigma$, otherwise $\Sigma$ is either an Abresch-Rosenberg $H-$surface or a Hopf cylinder. Then, from \eqref{lapla2}, we get
\begin{equation*}
C  =  \abs{S} (V + \lambda_{1}(L))+ \Delta |S| \geq   \abs{S}(V + \lambda_{1}(L)+ 2K).
\end{equation*}

Note that $C >  \abs{S}(V + \lambda_{1}(L) + 2 K)  $ at some point since, otherwise, it would implies  that $|S|$ is constant, which is a contradiction.

Suppose $\lambda_{1}(L) \geq -\inf_{\Sigma}\{V+2K\}$,  then $C \geq \abs{S}(V + \lambda_{1}(L) + 2 K) \geq 0$. Now, take $p \in \Sigma$ a fixed point and $R>0$. Denote by $r(x) = d(x,p)$ the distance function from $p$ and $B(p,R)$ the geodesic ball of radius $R$. Choose $R' < R$ and define $\phi$ as follows 

\begin{equation*}
\phi(x) =
\begin{cases}
 1 & \text{For $x$ such that $0 \leq r(x) \leq \frac{R'}{2}$. }\\
2 - \frac{2}{R'}r(x) & \text{For $x$ such that $\frac{R'}{2} < r(x) \leq R'$.}\\
0 & \text{For $x$ such that $R' < r(x) \leq R$.}\end{cases}
\end{equation*}

Observes that $\phi \in C^{1}_{0}(B(p,R'))$, $\overline{B(p,R')} \subset B(p,R)$ and $B(p,R)$ is a relatively compact set on $\Sigma$, then Lemma \ref{Funtes} implies 

\begin{equation*}
\int _{B(p,R)} \phi ^{2} \abs{S} C_{B(p,R)}\leq \int _{B(p,R)} \abs{S}^{2}\norm{\nabla \phi}^{2},
\end{equation*}and, since $\phi = 1$ on $B(p, \frac{R'}{2})$ and $C \geq 0$ on $\Sigma$, we get
\begin{equation*}
\int_{B(p,\frac{R'}{2})} \abs{S}C_{B(p,R)} \leq \int _{B(p,R)} \phi ^{2} \abs{S} C_{B(p,R)} \leq \frac{4}{(R')^{2}}\displaystyle{\int_{\{x \in \Sigma: \frac{R'}{2} < r(x) \leq R' \}} \abs{S}^{2}}. 
\end{equation*}

Hence, from the hypothesis that $\Sigma$ has finite Abresch-Rosenberg total curvature  and letting $R' \rightarrow \infty$ we obtain
$$\int_{\Sigma}\abs{S}C \leq 0.$$ 

The above equation implies that $C \equiv 0$ on $\Sigma$, hence $\abs{S}$ must be constant on $\Sigma$, which is a contradiction since we are assuming that $\abs{S}$ is not constant. Therefore $\lambda _{1}(L)$ must satisfy 
$$ \lambda _{1}(L) < - {\rm inf}_{\Sigma}\set{V+2K}.$$ 
\end{proof}


Obviously,  any closed $H-$surface in $\hmf$ has finite Abresch-Rosenberg total curvature. Consequently, from  \cite[Theorem 1]{To2} and Theorem \ref{Thm:Estimate1} we get an estimate of $\lambda_{1}(L)$ of any closed surface of $\hmf$, $H^{2}+\t^{2}\neq 0$. 

\begin{corollary}\label{Cor:Estimate1}
Let $\Sigma$ be a closed $H-$surface in $\hmf$, $H^{2}+\t^{2} \neq 0$. Denote by $\lambda_{1}(L)$ the first eigenvalue associated to the Schr\"{o}dinger Operator  $L := \Delta + V$, $V\in C^0 (\Sigma)$. Then, $\Sigma$ is either a rotationally symmetric $H-$sphere, a Hopf $H-$tori or 
\begin{equation}
\label{Esl}
\lambda _{1}(L) < - {\rm inf}_{\Sigma}\set{V+2K}.
\end{equation}
\end{corollary}

\subsection{Stability Operator}

Now, we obtain estimates for the most natural Schr\"{o}dinger operator of a complete $H-$surface in $\hmf$, the Stability (or Jacobi) operator 
$$ J = \Delta + (\abs{A}^{2} + {\rm Ric} (N)) ,$$where ${\rm Ric}(N)$ is the Ricci curvature of the ambient manifold in the normal direction. Hence, in this case, $V \equiv \abs{A}^{2} + {\rm Ric} (N)$.

Note that, since ${\rm Ric}( N) = (\kappa - 4\t^{2})|\mt|^{2} + 2\t^{2}$ and the Gauss equation \eqref{Gauss}, we have 
\begin{equation*}
\begin{split}
V+2K &= 4H^{2} +2(K-K_{e}) + (\kappa - 4\t^{2})|\mt|^{2}  +2 \t ^{2}\\
 &= 4 H^{2} + \kappa + (\kappa - 4\t^{2})\nu ^{2} .
\end{split}
\end{equation*}

Hence, Theorem \ref{Thm:Estimate1} and the above equality gives:
\begin{theorem}\label{Stability}
Let $\Sigma$ be a complete  two sided $H-$surface of finite Abresch-Rosenberg total curvature  in $\hmf$, $H^{2}+\t^{2} \neq 0$. 
\begin{itemize}
\item If $\kappa - 4\t ^{2} > 0$. Then, $\Sigma $ is either an Abresch-Rosenberg $H-$surface, a Hopf cylinder, or
\begin{equation*}
\lambda_{1} < - (4H^{2} + \kappa ).
\end{equation*}
\item If $\k-4\t^{2} < 0$. Then, $\Sigma$ is either an Abresch-Rosenberg $H-$surface, or 
\begin{equation*} 
\lambda_{1} <  - (4H^{2} + \kappa ) - (\kappa-4\t^{2}).
\end{equation*}
\end{itemize}
\end{theorem}

\begin{remark}
These estimates were obtained by Al\'{i}as-Mero\~{n}o-Ort\'{i}z \cite{AMO} for closed surfaces in $\hmf$. 
\end{remark}

\section{Pinching Theorems for $H-$surfaces in $\hmf$}

In this section we use the Simons' type formula \eqref{lapla} together with the Omori-Yau Maximum Principle to classify complete $H-$surfaces in $\hmf$ satisfying a pinching condition on its Abresch-Rosenberg fundamental form. First, we recall the Omori-Yau Maximum Principle for the reader convenience. 

\begin{theorem}[\cite{Y}]
Let $M$ be a complete Riemannian manifold with Ricci curvature bounded from below. If $u \in C^{\infty}(M)$ is bounded from above, then there exits a sequence of points $\{p_{j}\}_{j \in \mathbb{N}} \in M$ such that:
\begin{enumerate}
\item[1.] $\displaystyle{\lim_{j \rightarrow \infty} u(p_{j}) = \sup_{M}u}$.
\item[2.] $|\nabla u|(p_{j}) < \frac{1}{j}$.
\item[3.] $\Delta u(p_{j}) < \frac{1}{j} $.
\end{enumerate}
\end{theorem}






Second, we study the Simons' type formula \eqref{lapla} in the set of non-umbilical points of $S$. 

\begin{prop}\label{Prop:Polinomio}
Let $\Sigma$ be a $H-$surface in $\hmf$. Then, away from the umbilic points of $S$, it holds
\begin{equation}\label{lapla3}
\frac{1}{2}\Delta|S|^{2} \geq  |\nabla S|^{2} + |S|^{2}F(|S|),
\end{equation}where  $F(x) = -x^{2} + bx + a$ is the second degree polynomial given by 
\begin{equation*}
\begin{split}
a & = 2(\k-4\t^{2})+2(H^{2} + \t^{2})  -2(\k-4\t^{2})|\mt|^{2}-\alpha^{2}\frac{\abs{\mt}^{4}}{2} ,\\
b & =  - 2 \abs{\alpha} |\mt|^{2} ,
\end{split}
\end{equation*}where $\alpha = \frac{\k-4\t^{2}}{2\sqrt{H^{2}+\t^{2}}}$.
\end{prop}

\begin{proof}
From \eqref{eq2}, the Gaussian curvature can be written as
\begin{equation*}
K = (\kappa-4\t^{2})(1-\abs{\mt}^{2})+\t^{2}+H^{2}- \alpha^{2}\frac{\abs{\mt}^{4}}{4}-\frac{1}{2}\abs{S}^{2}-\alpha \la S\mt_{\theta},\mt_{\theta} \ra .
\end{equation*}

Since $\abs{\la S\mt_{\theta},\mt_{\theta} \ra} \leq \abs{S}\abs{\mt }^{2}$, substituting the above formulas into \eqref{lapla} yields
\begin{equation*}
\begin{split}
\frac{1}{2}\Delta|S|^{2} \geq |\nabla S|^{2} & + |S|^{2} \big(  2(\kappa-4\t^{2})(1-\abs{\mt}^{2})+2(\t^{2}+H^{2})- \alpha^{2}\frac{\abs{\mt}^{4}}{2}\\
    & \qquad \qquad -\abs{S}^{2}-2 \abs{\alpha }\abs{S}\abs{\mt}^{2} \big) ,
\end{split}
\end{equation*}as claimed.
\end{proof}

So, our next step is to study the first positive root $\bar x \in \r ^{+}$ of $F(x)$ so that $F(x)>0$ for all $x \in (0,\bar x)$. To do so, we set  $t  = \abs{\mt}^{2} \in [0,1]$
and hence, we can rewrite:
\begin{equation*}
\begin{split}
a(t) &= 2(\kappa-4\t^{2})  + 2(H^{2}+\t^{2}) -2(\kappa-4\t^{2})t -\frac{\alpha^{2}}{2}t^{2}.\\
b(t) &= -2 \abs{\alpha} t.
\end{split}
\end{equation*}

In order to obtain a positive real root, the coefficients of $F$ must hold $h(t)= b(t)^{2}+ 4a(t) > 0$ and $a(t)>0$ for all $t\in [0,1]$. This means that 
$$  (H^{2}+\t^{2}) + (\kappa-4\t^{2})(1-t)+ \frac{\alpha ^{2}}{4}t^{2} > 0 \text{ for all } t \in [0,1] ,$$and
$$  2(H^{2}+\t^{2}) + 2(\kappa-4\t^{2})(1-t)- \frac{\alpha ^{2}}{2}t^{2} > 0 \text{ for all } t \in [0,1].$$

\begin{prop}\label{Zero1}
Define $G: [0,1]  \rightarrow \mathbb{R}$ as
\begin{equation}\label{zero}
G(t) = \frac{b(t) + \sqrt{b(t)^{2}+4a(t)}}{2},
\end{equation}then, assuming $h(t)>0$ and $a(t)>0$ for all $t\in [0,1]$, the function $G(t)$ satisfies:
\begin{equation}\label{MinG}
{\rm min}_{t\in[0,1]} G(t) = {\rm min}\set{\sqrt{2}\sqrt{(H^{2}+\t^{2})+(\kappa -4\t^{2})} , -\abs{\alpha} +\sqrt{2(H^{2}+\t^{2}) +\frac{\alpha^{2}}{2}}}.
\end{equation}
\end{prop}
\begin{proof}
We compute the interior critical points of $G(t)$. To do so, we compute the derivative of $G$ at an interior point 
\begin{equation*}
\begin{split}
G'(t) &= \frac{b'(t)}{2}\left( 1+ \frac{b(t)}{\sqrt{b(t)^{2}+4a(t)}} \right)+ \frac{a'(t)}{\sqrt{(b(t))^{2}+4a(t)}}\\
    & =-\abs{\alpha}\left(\frac{b(t)+\sqrt{b(t)^{2}+4a(t)}}{\sqrt{b(t)^{2}+4a(t)}} \right) - \frac{\alpha ^{2}t + 2(\kappa -4 \t ^{2})}{\sqrt{(b(t))^{2}+4a(t)}} \\
    &= \frac{1}{\sqrt{b(t)^{2}+4a(t)}}\left(  -2\abs{\alpha} G(t)-\alpha ^{2}t- 2(\kappa -4 \t ^{2}) \right)
\end{split}
\end{equation*}

Assume that there exists $\bar t \in (0,1)$ so that $G'(\bar t) =0$, then the above equation implies that the function $ \Psi (t) := -2\abs{\alpha} G(t)-\alpha ^{2}t- 2(\kappa -4 \t ^{2}) $ satisfies $\Psi (\bar t) =0$ and $\Psi ' (\bar t) =-\alpha ^2 \bar t$. Moreover, observe that
$$ G'' (t) = R(t) \Psi (t) + \frac{\Psi ' (t)}{\sqrt{b(t)^{2}+4a(t)}} ,$$for some smooth function $R:(0,1)\to \r ^+$. Hence
$$ G'' (\bar t) =  \frac{\Psi (\bar t)}{\sqrt{b(\bar t)^{2}+4a(\bar t)}} =  - \frac{\alpha ^2 \bar t}{\sqrt{b(t)^{2}+4a(t)}} <0.$$

Therefore, $G$ does not have an interior minimum. Therefore, in any case, ${\rm min}_{t\in [0,1]} G(t) = {\rm min}\set{G(0) ,G(1)}$. 
\end{proof}

Next, we compute the minimum of $G$. To do so, we will distinguish two cases depending on the sign of $\kappa - 4 \t ^2$.

\begin{quote}
{\bf Case 1: $\k-4\t^{2} > 0$.} In this case, we assume $4(H^{2}+\t^{2})>\kappa -4 \t^{2}$. Then, 
$$  h(t)>0 \text{ and } a(t) >0 \text{ for all } t \in [0,1].$$

Therefore, Proposition \ref{Zero1} implies
$$ {\rm min}_{t\in[0,1]} G(t) = \sqrt{2}\sqrt{(H^{2}+\t^{2})+(\kappa -4\t^{2})} >0 .$$
\end{quote}

Now, we are ready to show a pinching theorem for complete $H-$ surfaces in $\hmf$ when $\k -4\t^{2}>0$.

\begin{theorem}\label{ap1}
Let $\Sigma$  be a complete immersed $H-$surface in $\hmf$, $\kappa - 4\tau^{2} >0$. Assume  that $4(H^{2}+\t^{2}) > \kappa - 4\tau^{2}$  and 
$$\sup_{\Sigma} |S| <  \sqrt{2}\sqrt{(H^{2}+\t^{2})+(\kappa -4\t^{2})},$$
where $S$ is the traceless Abresch-Rosenberg shape operator. Then, $\Sigma$ is an Abresch-Rosenberg surface in  $\hmf$.  

Moreover, if 
$$\sup_{\Sigma} |S| =  \sqrt{2}\sqrt{(H^{2}+\t^{2})+(\kappa -4\t^{2})}$$and there exists one point $p \in \Sigma$ such that  $\abs{S(p)}=\sup_{\Sigma} |S|$, then $\Sigma$ is a Hopf cylinder.
\end{theorem}

\begin{proof} 
Assume $\abs{S}$ is not identically zero. Set $d := F({\rm sup}_{p\in \Sigma} \abs{S})$. From Proposition \ref{Zero1} we have 
$$ F(\abs{S}(p)) \geq d >0 . $$

Hence, Proposition \ref{Prop:Polinomio} implies that 
\begin{equation}\label{C1}
\Delta |S|^{2} \geq d \abs{S}^{2}
\end{equation}away from the umbilic points of $S$.

Thus, since $\abs{S}$ is bounded by hypothesis, Corollary \ref{BoundedC} implies that the Ricci curvature of $\Sigma$ (nothing but the Gaussian curvature) is bounded from below. Hence, the Omori-Yau Maximum Principle implies that there exists a sequence $\{ p_{j} \}_{j \in \mathbb{N}} \in \Sigma$ such that:
$$\lim _{j \rightarrow \infty} |S|^{2}(p_{j}) = \sup_{\Sigma}|S|^{2} \ \ and \ \ \ \ \ \lim_{j \rightarrow \infty} \Delta |S|^{2}(p_{j}) \leq 0.,$$which contradicts \eqref{C1}. Hence, $|S|= 0$ on $\Sigma$, that is, $\Sigma$ is an Abresch-Rosenberg surface.

Now, assume that $$\sup_{\Sigma} |S| =  \sqrt{2}\sqrt{(H^{2}+\t^{2})+(\kappa -4\t^{2})}$$and there exists one point $p \in \Sigma$ such that  $\abs{S(p)}=\sup_{\Sigma} |S| >0$. So, from Proposition \ref{Zero1}, there is a neighborhood $\Omega$ of $p$ so that
$$\Delta \abs{S} \geq 0.$$ 

Then, the Interior Maximum Principle implies that $\abs{S} $ is constant (non zero) on $\Omega$, and hence constant on $\Sigma$. Thus, Lemma \ref{Lem:qconstant} implies that $\Sigma$ is a Hopf cylinder.
\end{proof}

\begin{quote}
{\bf Case 2: $\k-4\t^{2} < 0$.}  In this case, we assume $H^{2}+\t^{2} > \abs{\k-4\t^{2}}$. Then, 
$$  h(t)>0 \text{ and } a(t) >0 \text{ for all } t \in [0,1].$$

Therefore, Proposition \ref{Zero1} implies
$$ {\rm min}_{t\in[0,1]} G(t) = -\abs{\alpha} +\sqrt{2(H^{2}+\t^{2}) +\frac{\alpha^{2}}{2}} >0$$
\end{quote}

So, arguing as in Theorem \ref{ap1}, we establish (without proof) a pinching theorem for complete $H$- surfaces in $\hmf$ when $\k - 4\t^{2}<0$. 

\begin{theorem}\label{ap2}
Let $\Sigma$  be a complete immersed $H-$surface in $\hmf$, $\k - 4\tau^{2} <0$. Assume that $H^{2}+\t^{2} > \abs{\k-4\t^{2}}$ and 
$$\sup_{\Sigma} |S| <  -\abs{\alpha} +\sqrt{2(H^{2}+\t^{2}) +\frac{\alpha^{2}}{2}},$$ where $S$ is the traceless Abresch-Rosenberg shape operator. Then, $\Sigma$ is  an Abresch-Rosenberg surface of $\hmf$.

Moreover, if 
$$\sup_{\Sigma} |S| =   -\abs{\alpha} +\sqrt{2(H^{2}+\t^{2}) +\frac{\alpha^{2}}{2}}$$and there exists one point $p \in \Sigma$ such that  $\abs{S(p)}=\sup_{\Sigma} |S|$, then $\Sigma$ is a Hopf cylinder.
\end{theorem}

\section*{Acknowledgments}

The first author, Jos\'{e} M. Espinar, is partially supported by Spanish MEC-FEDER Grant
MTM2013-43970-P; CNPq-Brazil Grants 405732/2013-9 and 14/2012 - Universal, Grant
302669/2011-6 - Produtividade; FAPERJ Grant 25/2014 - Jovem Cientista de Nosso Estado.

The second author, Haimer A. Trejos, is supported by CNPq-Brazil.

\end{document}